\documentclass[12pt]{article}

\usepackage{amsmath,amsthm,amssymb,graphicx,tikz,mathtools,enumitem}
\usepackage{tikz-cd}
\usepackage[margin=1in]{geometry}
\usepackage{mathrsfs}
\usepackage[all]{xy}
\usepackage{amsfonts}
\usepackage{latexsym}
\usepackage{subfigure}
\usepackage{fullpage}
\usepackage{float}
\usepackage{xypic}
\usepackage{bbm}
\usepackage{xcolor}
\usepackage{tikz}

\usepackage[percent]{overpic}

\DeclareMathOperator{\AKh}{AKh}

\DeclareMathOperator{\CKh}{CKh}

\DeclareMathOperator{\rank}{rank}
\DeclareMathOperator{\lk}{lk}
\DeclareMathOperator{\wrap}{wrap}

\newtheorem{thm}{Theorem}[section]

\newtheorem{cor}[thm]{Corollary}
\newtheorem{conj}[thm]{Conjecture}

\theoremstyle{definition}
\newtheorem{rem}[thm]{Remark}
\newtheorem{defin}[thm]{Definition}

\usepackage{hyperref}
\hypersetup{pdftex,colorlinks=true,allcolors=blue}

\begin{document}
\title{Annular Khovanov homology and meridional disks}
\author{Gage Martin}
\date{}
\maketitle

\begin{abstract} 

We exhibit infinite families of annular links for which the maximum non-zero annular Khovanov grading grows infinitely large but the maximum non-zero annular Floer-theoretic gradings are bounded. We also show this phenomenon exists at the decategorified level for some of the infinite families. Our computations provide further evidence for the wrapping conjecture of Hoste-Przytycki and its categorified analogue. Additionally, we show that certain satellite operations cannot be used to construct counterexamples to the categorified wrapping conjecture. We also extend the Batson-Seed link splitting spectral sequence to the setting of annular Khovanov homology.
\end{abstract}
\section{Introduction} 

Jones brought new ideas to the field of low-dimensional topology with his construction of the Jones polynomial for links in $S^3$~\cite{jones_polynomial_1985}. A decade and a half later, Khovanov categorified the Jones polynomial with Khovanov homology, a bi-graded abelian group whose graded Euler characteristic recovers the Jones polynomial~\cite{khovanov_categorification_1999}.

Annular Khovanov homology was defined by Asaeda-Przytycki-Sikora~\cite{asaeda_categorification_2004}, who introduced a version of Khovanov homology for links in thickened surfaces. They also showed that the graded Euler characteristic of annular Khovanov homology is the Kauffman skein bracket of annular links defined by Hoste-Przytycki~\cite{hoste_invariant_1989}.

Since the introduction of these invariants, some natural questions have arisen about what, if any, relationship exists between topological properties of links and their Jones polynomial or Khovanov homology. Conjecturally there is a relationship between the Kauffman skein bracket of annular links and certain embedded disks in the annular complement of the link.

\begin{conj}[The wrapping conjecture~\cite{hoste_2_1995}]\label{WrappingConjecture} 
Let $L$ be an annular link in $A \times I$ and let $n$ be the minimal intersection of $L$ with a meridional disk. Then the maximal non-zero annular degree of the Kauffman skein bracket of $L$ is $n$.
\end{conj}

When the conjecture was stated, Hoste-Przytycki give an argument that the wrapping conjecture holds for any annular link with a $\pm$-adequately wrapped diagram~\cite{hoste_2_1995}.

Since the graded Euler characteristic of annular Khovanov homology is the Kauffman skein bracket, an immediate consequence of Conjecture~\ref{WrappingConjecture} would be:

\begin{conj}[The categorified wrapping conjecture]
Let $L$ be an annular link in $A \times I $ and let $n$ be the minimal intersection of $L$ with a meridional disk. Then the maximal non-zero annular grading of the annular Khovanov homology of $L$ is $n$.
\end{conj}

The statement of the categorified wrapping conjecture first appeared in a talk by Eli Grigsby at the MSRI semester-long program {\em Homology theories of knots and links} in spring 2010, where she claimed a proof of the conjecture relying on the spectral sequence to the Floer homology of the double-branched cover~\cite{roberts_knot_2013}~\cite{grigsby_khovanov_2010} and Juhasz's Thurston norm detection results~\cite{juhasz_floer_2008}. In the week after the talk, Matt Hedden and Stephan Wehrli independently found examples that are a subset of the family that appears here in Theorems~\ref{AKhthm} and~\ref{KBthm}, which proved that the argument she suggested did not work. These examples were not pursued further until recent work of Xie~\cite{xie_instantons_2018} relating annular Khovanov homology to instanton Floer homology sparked new interest in the topic. 

Some results about annular gradings where $\AKh(L)$ is non-zero were already known. Grigsby-Ni showed that the categorified wrapping conjecture holds for string links~\cite{grigsby_sutured_2014}. If we allow surfaces of higher genus, work of Xie~\cite{xie_instantons_2018} and Xie-Zhang~\cite{xie_instanton_2019} shows that $\AKh(L)$ is non-zero in the annular grading of the generalized Thurston norm for all meridional surfaces. However, we will see examples of annular links where the generalized Thurston norm is much smaller than the minimal intersection number with a meridional disk.

In this paper we verify that categorified wrapping conjecture holds for some new families of annular links. For a subset of these links we also show that the wrapping conjecture holds on the decategorified level.

\newtheorem*{GenCab}{Theorem~\ref{generalcable}}
\begin{GenCab}
Let $L$ be an $n$-component annular link for which the categorified wrapping conjecture holds and let $L_{m,\textbf{s}}$ be the link where the $i$-th component of $L$ is replaced with a link $s_i$ obtained as the closure of a tangle $T_i$ consisting of an $m$-string link and possibly additional closed components. Then the categorified wrapping conjecture also holds for the cable $L_{m,\textbf{s}}$.
\end{GenCab}

\newtheorem*{AKhTHM}{Theorem~\ref{AKhthm}}
\begin{AKhTHM}
Let $L^n$ be the annular link built by vertically stacking $n$ copies of tangle $J$ from Figure~\ref{tangle} and then taking the annular closure. Also, let $L^{n}_{\textbf{m},\textbf{s}}$ be the link built by replacing the $i$-th component of $L^n$ with a link $s_i$ obtained as the closure of a tangle $T_i$ consisting of an $m_i$-string link and possibly additional closed components. Then the categorified wrapping conjecture holds for the cable $L^{n}_{\textbf{m},\textbf{s}}$.
\end{AKhTHM}

\newtheorem*{AKhBraid}{Theorem~\ref{AKhBraidthm}}
\begin{AKhBraid}
Let $L^n$ be the annular link built by vertically stacking $n$ copies of tangle $J$ from Figure~\ref{tangle} and then taking the annular closure. Also, let $L^{n}_{m,\beta_1, \ldots, \beta_n}$ be the link built by replacing the $i$-th component of $L^n$ with the closure of the $m$-braid $\beta_i$. Then the categorified wrapping conjecture holds for the cable $L^{n}_{m,\beta_1, \ldots, \beta_n}$.
\end{AKhBraid}

\newtheorem*{AKhWh}{Theorem~\ref{AKhWhThm}}
\begin{AKhWh}

Let $K^n$ be the iterated Whitehead double of the annular link $L^1$ obtained as the annular closure of the tangle $J$ from Figure~\ref{tangle} and let $K^n_{m,\beta}$ be the annular knot obtained by replacing $K^n$ with the closure of the $m$-braid $\beta$. Then the categorified wrapping conjecture holds for $K^n_{m,\beta}$.

\end{AKhWh}

\newtheorem*{KBTHM}{Theorem~\ref{KBthm}}
\begin{KBTHM}
Let $L^n$ be the annular link built by vertically stacking $n$ copies of tangle $J$ from Figure~\ref{tangle} and then taking the annular closure. Also, let $L^{n}_{m,\beta_1, \ldots, \beta_n}$ be the link built by replacing the $i$-th component of $L^n$ with the closure of the $m$-braid $\beta_i$. Then the wrapping conjecture holds for the cable $L^{n}_{m,\beta_1, \ldots, \beta_n}$.
\end{KBTHM}

\newtheorem*{KBWh}{Theorem~\ref{KBWhThm}}
\begin{KBWh}

Let $K^n$ be the iterated Whitehead double of the annular link $L^1$ obtained as the annular closure of the tangle $J$ from Figure~\ref{tangle} and let $K^n_{m,\beta}$ be the annular knot obtained by replacing $K^n$ with the closure of the $m$-braid $\beta$. Then the wrapping conjecture holds for $K^n_{m,\beta}$.

\end{KBWh}

\begin{figure}

  \centering
    \includegraphics[width=0.3\textwidth]{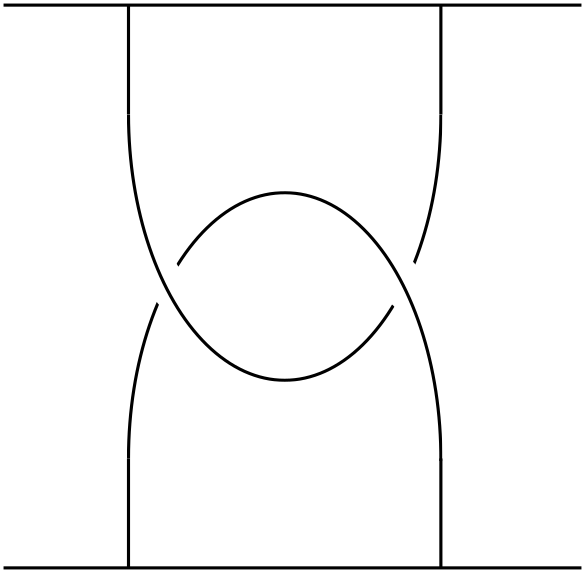}
      \caption{The tangle $J$ we use to construct some annular links.}\label{tangle}
\end{figure}

The links considered in Theorems~\ref{AKhthm},~\ref{AKhBraidthm},~\ref{AKhWhThm},~\ref{KBthm}, and~\ref{KBWhThm} are examples of links where the generalized Thurston norm is much smaller than the minimal intersection of $L$ with a meridional disk. The generalized Thurston norm is no larger than two for any of these links because there is an embedded meridional torus which does not intersect the link. However, for these examples, the minimal intersection with a meridional disk can be made arbitrarily large. This gives rise to a difference between the Kauffman bracket and annular Khovanov homology on the one hand and the multi-variable Alexander polynomial and various annular Floer theories on the other hand.

\begin{cor} 
There is an infinite family of links $L_i$ such that the maximal non-zero annular grading for the annular Khovanov homology of the links $L_i$ grows infinitely large but the maximal non-zero annular grading for annular instanton homology and annular link Floer homology of the links $L_i$ is bounded.
\end{cor}

\begin{cor}
There is an infinite family of links $L_i$ such that the maximal non-zero annular grading for the Kauffman bracket of the links $L_i$ grows infinitely large but the maximal non-zero annular grading for the multi-variable Alexander polynomial  of the links $L_i$ is bounded.
\end{cor}

Recalling the relationship between annular Khovanov homology and knot Floer homology using double branched covers~\cite{roberts_knot_2013, grigsby_khovanov_2010}, we also have the following differences.

\begin{cor}
There is an infinite family of links $L_i$ such that the maximal non-zero annular grading for the annular Khovanov homology of the links $L_i$ grows infinitely large but the maximal non-zero Alexander grading for an associated link in $\Sigma(L_i)$ is bounded.

\end{cor}

\begin{cor}
There is an infinite family of links $L_i$ such that the maximal non-zero annular grading for the Kauffman bracket of the links $L_i$ grows infinitely large but the maximal non-zero degree of the Alexander polynomial for an associated link in $\Sigma(L_i)$ is bounded.
\end{cor}

The proofs of Theorems~\ref{KBthm} and~\ref{KBWhThm} are entirely diagrammatic and involve understanding a specific resolution of the links in question. To prove Theorems~\ref{generalcable},~\ref{AKhthm},~\ref{AKhBraidthm}, and~\ref{AKhWhThm} we extend the Batson-Seed link splitting spectral sequence~\cite{batson_link-splitting_2015} to the annular setting.

\newtheorem*{ABS}{Theorem~\ref{annularBS}}
\begin{ABS}
Let $L$ be an annular link and $R$ a ring. Choose weights $w_c \in R$ for each component $c$ of $L$. Then there is a spectral sequence with pages $E_k(L,w)$, and $$E_1(L,w) \cong \AKh(L;R)$$
If the difference $w_c - w_d$ is invertible in $R$ for each pair of components $c$ and $d$ with distinct weight, then the spectral sequence converges to $$\AKh\left( \coprod_{r\in R}L^{(r)};R \right) $$ where $L^{(r)}$ denotes the sub-link of $L$ consisting of those components with weight $r$.
\end{ABS}

The organization of the paper is as follows. In Section~\ref{Background} we give relevant background on annular links, the Kauffman bracket, annular Khovanov homology, and the Batson-Seed link splitting spectral sequence in $S^3$. In Section~\ref{AnnularLinkSplit} we extend the Batson-Seed construction to the annular setting and prove Theorem~\ref{annularBS}. In Section~\ref{Applications} we apply the annular link splitting spectral sequence to prove Theorems~\ref{generalcable},~\ref{AKhthm},~\ref{AKhBraidthm}, and~\ref{AKhWhThm}. Finally in Section~\ref{KBSection} we turn our attention to the Kauffman bracket and prove Theorems~\ref{KBthm} and~\ref{KBWhThm}.

\subsection*{Acknowledgements}
The author would like to thank Eli Grigsby, Patrick Orson, and J\'{o}zef Przytycki for helpful conversations.

\section{Background}\label{Background}

An $n$-component annular link $L$ is an embedding of $n$ circles into the thickened annulus $A \times I$ considered up to ambient isotopy. Alternatively, an $n$-component annular link is an $n+1$-component link in $S^3$ with a distinguished unknotted component. A diagram of an annular link $L$ is a choice of a generic projection of $L$ onto $A \times \{0\}$ which records crossing information.

Given a crossing of an annular link, there are two possible resolutions of the crossing. These are referred to as the $0$-resolution and the $1$-resolution and are shown in Figure~\ref{10res}. Notice that it is possible to add a band to transform the 0-resolution into the 1-resolution or to transform the 1-resolution back into the 0-resolution. The locations for attaching the bands is indicated by dashed lines in Figure~\ref{10res}. 

\begin{figure}

  \centering
    \includegraphics[width=0.65\textwidth]{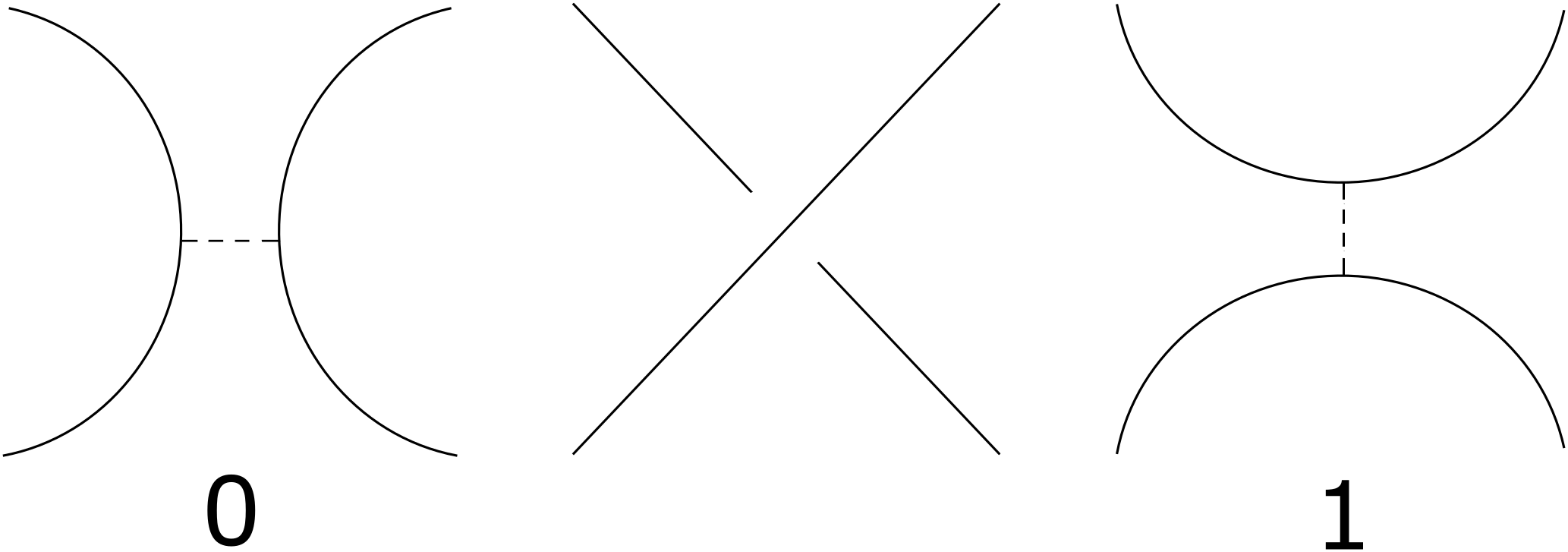}
      \caption{The 0-resolution and 1-resolutions of a crossing. The dotted lines indicate where to attach bands to change between the 0-resolution and the 1-resolution.}\label{10res}
\end{figure}

Asaeda-Przytycki-Sikora constructed a version of Khovanov homology for annular links which is now referred to as annular Khovanov homology. Annular Khovanov homology categorifies the Kauffman bracket of annular links~\cite{asaeda_categorification_2004}. Since the introduction of the theory, the main application of annular Khovanov homology to low-dimensional topology is the study of braid closures~\cite{BG, binns_knot_2020, grigsby_sutured_2014,hubbard_sutured_2017,hubbard_annular_2016}. Additionally there have been spectral sequences constructed relating annular Khovanov homology to various Floer theories~\cite{roberts_knot_2013, grigsby_khovanov_2010, xie_instantons_2018}.

The annular Khovanov homology of an annular link $L$ is constructed by taking a diagram for $L$ in the annulus and constructing a cube of resolutions from the diagram, assigning a triply graded vector space to each complete resolution, and assigning linear maps to the edges of the cube. For our applications, the details of the vectors spaces and maps are not needed but a full definition of this invariant appears in~\cite{asaeda_categorification_2004} where it was introduced. 

\subsection{Batson-Seed link splitting spectral sequence} 
 In~\cite{batson_link-splitting_2015}, Batson-Seed constructed a link-splitting spectral sequence for Khovanov homology of links in $S^3$. In Section~\ref{AnnularLinkSplit} we will verify that their arguments also work for constructing a similar spectral sequence for annular Khovanov homology. Here we will briefly recall some details of their construction relevant to Section~\ref{AnnularLinkSplit}.
 
 Batson-Seed construct their spectral sequence by taking the Khovanov chain complex and perturbing it by adding an additional differential $\partial^{BS}$. This differential does not respect the $i$ or $j$-gradings on the chain complex but does respect an $\ell$-grading defined as the difference $i - j$. Additionally, the entire chain complex with the perturbation is $g$-filtered where for an $n$-component link, the $g$-filtration is defined as $\frac{ j - n} { 2}$. This filtration is what induces the spectral sequence from the Khovanov homology of a link to the Khovanov homology of a splitting of the link.
 
 The construction of the perturbed differential requires a choice of sign assignment $s$ on a diagram of the link but Batson-Seed show that the filtered chain homotopy type of the construction does not depend on these choices. Additionally the construction requires that every component of $L$ be given a weight $w_i$,  when the difference $w_i - w_j$ between any two non-equal weights is a multiplicative unit then the spectral sequence converges to the Khovanov homology of a link built by taking disjoint union of the sublinks $L_i$ consisting of all the components weighted by $w_i$. For the applications in this paper we will be working over the field $\mathbb{C}$ so the unit condition is always satisfied.

\subsection{Sutured Khovanov homology of balanced tangles in $D\times I$}

The computations in Section~\ref{KBSection} require working with Khovanov-type invariants of balanced tangles. We briefly recall some information about these invariants and their relationship with annular Khovanov homology here.

A tangle $T$ is an embedding of some number of circles and intervals into $D^2 \times I$ so that the boundary of  $T$ is a subset of $D^2 \times \{0\} \cup D^2 \times \{1\}$. A tangle is balanced if the intersection number of $T$ with $D^2 \times \{0\}$ is the same as the intersection number of $T$ with $D^2 \times \{1\}$. We can assume that the intersections of $T$ with the two disks happen in the same points.

As with the other Khovanov-type invariants, the sutured Khovanov homology of a balanced tangle $T$ is constructed by taking a diagram for $T$ and constructing a cube of resolution made up of flat tangles, tangle diagrams with no crossings. Each balanced tangle is then replaced with a graded vector space and the edges of the cube are replaced with linear maps. For our applications, the details of the vectors spaces and maps are not needed but a full definition of this invariant appears in~\cite{elisenda_grigsby_colored_2010} where it was introduced.

Given a balanced tangle $T$, we can construct an annular link by identifying $D^2 \times \{0\}$ with $D^2 \times \{1\}$ via the identity map which we will call the annular closure of the tangle. Reversing this process, it is possible to construct a balanced tangle from an annular link $L$ by decomposing or ``cutting open" $A \times I$ along a meridional disk that $L$ intersects transversely.

There is a relationship between the annular Khovanov homology of $L$ and the sutured Khovanov homology of a tangle $T$ obtained by decomposing $A \times I$ along a meridional disk $D$.

\begin{thm}[\cite{grigsby_khovanov_2010}]
Let $L$ be an annular link and let $T$ be the tangle obtained from $L$ by decomposing along a meridional disk $D$. Then the annular Khovanov homology of $L$ in $k$-grading $w$ is isomorphic to the sutured Khovanov homology of $T$ where $w$ is the number of intersection points of $L$ with $D$.
\end{thm}

Because the wrapping conjecture relates to a specific annular grading, we will use ideas from the sutured Khovanov homology of balanced tangles in parts of Section~\ref{KBSection} to assist with computations.

\section{Annular link splitting spectral sequence}\label{AnnularLinkSplit}

In this section we replicate many of the arguments from~\cite{batson_link-splitting_2015} to verify the existence of an annular link splitting spectral sequence and show that it has similar properties to the non-annular version.

\begin{thm}\label{annularBS}
Let $L$ be an annular link and $R$ a ring. Choose weights $w_c \in R$ for each component $c$ of $L$. Then there is a spectral sequence with pages $E_k(L,w)$, and $$E_1(L,w) \cong \AKh(L;R)$$
If the difference $w_c - w_d$ is invertible in $R$ for each pair of components $c$ and $d$ with distinct weight, then the spectral sequence converges to $$\AKh\left( \coprod_{r\in R}L^{(r)};R \right) $$ where $L^{(r)}$ denotes the sub-link of $L$ consisting of those components with weight $r$.
\end{thm}

\begin{proof}
Considering the annular link $L$ as a link in $S^3$ and choosing a diagram $D$ we can associate to it an $\ell$-graded, $g$-filtered, and $k$-filtered chain complex $\CKh(D;R)$ with the Khovanov differential $\partial$ and the Batson-Seed perturbation $\partial^{BS}$. Both of these differentials decompose into a portion that preserves the $k$-filtration, which we call $\partial_{0}$ and $ \partial_{0}^{BS}$, and a portion that lowers the $k$-filtration by $2$, which we call $\partial_{-}$ and $ \partial_{-}^{BS}$. Decomposing the relations $\partial^2 = 0$, $(\partial^{BS})^2 = 0$, and $\partial \partial^{BS} + \partial^{BS} \partial = 0$ into their $k$-homogeneous components immediately gives that $\partial_0^2 = 0$, $\partial_0^{BS} = 0$ and $\partial_0 \partial_0^{BS} + \partial_0^{BS} \partial_0 = 0$. This shows that equipping the chain complex for $L$ only with these differentials would give a bi-graded chain complex by gradings $\ell$ and $k$ which is also filtered by the $g$-filtration. Following the notation from~\cite{batson_link-splitting_2015}, we refer to this chain complex as $AC(D,w,s)$ where $w$ represents the weighting of the components and $s$ represents a choice of sign assignment.

In Section~2.3 of~\cite{batson_link-splitting_2015}, to show that the filtered chain homotopy type of their construction did not depend on the choice of sign assignment, Batson-Seed construct an explicit chain map giving the equivalence. Their map preserves the $k$-grading so it also shows that the filtered chain homotopy type of $AC(D,w,s)$ does not depend on the sign assignment. Similarly the chain maps in Proposition~2.3 of~\cite{batson_link-splitting_2015} used to show that the relatively $\ell$-graded total homology is unchanged by crossing changes also preserve the $k$-grading so their argument also applies in the annular setting.

The arguments to show that the filtered chain homotopy type does not depend on the choice of diagram in Section~2.3 of~\cite{batson_link-splitting_2015} work in the annular setting as well. The arguments in~\cite{batson_link-splitting_2015} consider local diagrams for the Reidemeister moves, resolve the local diagrams, construct some local cancelations and then produce an isomorphism on the level of a diagrammatic chain complex. The local nature of the arguments ensures that they also will work in the annular setting to show invariance under the Reidemeister moves. 

\end{proof}

The existence of this annular link splitting spectral sequence gives the following rank inequality as an immediate consequence. The proof follows exactly as in the proof of~Corollary~3.4 in~\cite{batson_link-splitting_2015}.

\begin{cor}
Let $\mathbb{F}$ be any field, and let $L$ be an annular link with components $K_1, \ldots, K_m$. Then $$ \rank^{\ell,k} \AKh^{*}(L,\mathbb{F}) \geq  \rank^{\ell+t,k} \otimes_{c=1}^m\AKh^{*}(K_c,\mathbb{F}) $$ where each side is bi-graded by $\ell,k$ and the shift $t$ is given by $$t = \sum_{c \leq d} 2\cdot \lk(L_c , L_d) $$ where $L_c$ and $L_d$ are components of the link $L$.
\end{cor}

As in the non-annular version, this annular link splitting spectral sequence can provide lower bounds on the splitting number of a link. We state the bound here but we will not use it in the rest of this paper. The proof of the bound is the same as the proof of Theorem~1.2 in~\cite{batson_link-splitting_2015}.

\begin{defin}
We say that an $n$-component link $L$ is an annular split link if after isotopy of the link in $A \times I$ it is possible to find numbers $t_1,\ldots ,t_{n-1} \in I$ such that the surfaces $(S^1 \times t_i) \times I$ in $A\times I$ separate the components of $L$.
\end{defin}

\begin{defin}
The annular splitting number of an annular link $L$ is the minimum number of times different components of the link must be passed through one another to obtain an annular split link.
\end{defin}

\begin{thm}
Let $L$ be an annular link and let $w_c \in R$ be a set of component weights such that $w_c - w_d$ is invertible for each pair of components $c$ and $d$. Let $b(L,w)$ be the largest integer $k$ such that $E_k \not= E_\infty(L,w)$. Then $b(L,w)$ is less than or equal to the annular splitting number of $L$.
\end{thm}

\section{Link splitting and the maximal annular grading}\label{Applications} 

Now we look at some applications of the spectral sequence from Theorem~\ref{annularBS} to verifying the categorified wrapping conjecture for some families of examples. The idea of the applications is to consider a specific splitting of the link so that it is easier to show the annular Khovanov homology of the resulting splitting is non-zero in the desired annular grading.

\begin{thm}\label{generalcable} 
Let $L$ be an $n$-component annular link for which the categorified wrapping conjecture holds and let $L_{m,\textbf{s}}$ be the link where the $i$-th component of $L$ is replaced with a link $s_i$ obtained as the closure of a tangle $T_i$ consisting of an $m$-string link and possibly additional closed components. Then the categorified wrapping conjecture also holds for the cable $L_{m,\textbf{s}}$.
\end{thm}

\begin{proof}
Notice that if $D$ is a disk which intersects $L$ in $\wrap(L)$ points, then $D$ intersects $L_{m,\textbf{s}}$ in $m \cdot \wrap(L)$ points. So we know $\wrap(L_{m,\textbf{f}}) \leq m \cdot \wrap(L)$. Now we will show that $\AKh(L_{m,\textbf{f}})$ is non-zero in $k$-grading $m \cdot \wrap(L)$. Together this will show that the categorified wrapping conjecture holds for $L_{m,\textbf{s}}$.

Consider the link $L_{m,\textbf{s}}$ and choose $m+1$ distinct weights $w_1, \ldots, w_{m+1} \in \mathbb{C}$ and weight the link $L_{m,\textbf{s}}$ so that for each individual satellite, the $m$ components of the string link closure are weighted differently using weights $w_1,\ldots, w_m$ and then all the other components are weighted with $w_{m+1}$.  Using these weights, the link splitting spectral sequence converges to the $k,\ell$ bi-graded annular Khovanov homology of a link built from the disjoint union of $m$ copies of $L$ by taking additional disjoint unions and connected sums with links in $S^3$.

The annular Khovanov homology of the disjoint union of $m$ copies of $L$ is non-zero in the $k$-grading $m \cdot \wrap(L)$ and taking disjoint unions and connected sums with links in $S^3$ does not change the maximal non-zero annular grading. For disjoint unions this is immediate from the decomposition as a tensor product and for connected sums this is the content of~\cite[Lemma~3.5]{grigsby_sutured_2014}. Then rank inequality from the annular link splitting spectral sequence ensures that $\AKh(L_{m,\textbf{s}})$ is non-zero in $k$-grading $m \cdot \wrap(L)$ as well.
\end{proof}

\begin{thm}\label{AKhthm} 

Let $L^n$ be the annular link built by vertically stacking $n$ copies of tangle $J$ from Figure~\ref{tangle} and then taking the annular closure. Also, let $L^{n}_{\textbf{m},\textbf{s}}$ be the link built by replacing the $i$-th component of $L^n$ with a link $s_i$ obtained as the closure of a tangle $T_i$ consisting of an $m_i$-string link and possibly additional closed components. Then the categorified wrapping conjecture holds for the cable $L^{n}_{\textbf{m},\textbf{s}}$.

\end{thm}

\begin{rem} 
Notice that Theorem~\ref{AKhthm} does not follow immediately from Theorem~\ref{generalcable} because here we are allowing the different components of the link to be replaced by links built from string links of varying numbers of strands.
\end{rem}

\begin{proof}

First notice the fact that the annular Khovanov homology of $L^n$ is non-zero in $k$-grading 2 follows immediately from the existence of a spectral sequence to annular instanton Floer homology~\cite{xie_instantons_2018} and that this theory is known to detect the Thurston norm of meridional surfaces~\cite{xie_instanton_2019}.

Let $p=m_j$ be the minimum over all the $m_i$, then there is a disk which intersects $L^{n}_{\textbf{m},\textbf{f}}$ in $2p$ points. Now we choose weights $w_1, \ldots , w_p, w_{p+1}\in \mathbb{C}$ and weight the components of $L^{n}_{\textbf{m},\textbf{s}}$ so that the first $p$ components of each string link are weighted by $w_1, \ldots , w_p$ and all remaining components are weighted by the final weight $ w_{p+1}$. 

Using these weights, the link splitting spectral sequence converges to the $k,\ell$ bi-graded annular Khovanov homology of a link built from the disjoint union of $p$ copies of $L^n$ by taking additional disjoint unions and connected sums with links in $S^3$. Because the annular Khovanov homology of $L^n$ is non-zero in $k$-grading 2, arguments from the proof of Theorem~\ref{generalcable} show the annular Khovanov homology of a link built from the disjoint union of $p$ copies of $L^n$ by taking additional disjoint unions and connected sums with links in $S^3$ is non-zero in grading $2p$. Then rank inequality from the annular link splitting spectral sequence ensures that $\AKh(L_{\textbf{m},\textbf{s}})$ is non-zero in $k$-grading $2p$ as well.

\end{proof}

There are other families of annular links that we can also show satisfy the categorified wrapping conjecture. The arguments work on the decategorified level and so the proofs of these theorems are delayed until Section~\ref{KBSection}. For completeness, we state that the categorified wrapping conjecture holds for these families as well.

\begin{thm}\label{AKhBraidthm}
Let $L^n$ be the annular link built by vertically stacking $n$ copies of tangle $J$ from Figure~\ref{tangle} and then taking the annular closure. Also, let $L^{n}_{m,\beta_1, \ldots, \beta_n}$ be the link built by replacing the $i$-th component of $L^n$ with the closure of the $m$-braid $\beta_i$. Then the categorified wrapping conjecture holds for the cable $L^{n}_{m,\beta_1, \ldots, \beta_n}$.
\end{thm}

\begin{thm}\label{AKhWhThm}

Let $K^n$ be the iterated Whitehead double of the annular link $L^1$ obtained as the annular closure of the tangle $J$ from Figure~\ref{tangle} and let $K^n_{m,\beta}$ be the annular knot obtained by replacing $K^n$ with the closure of the $m$-braid $\beta$. Then the categorified wrapping conjecture holds for $K^n_{m,\beta}$.

\end{thm}

\section{The maximal annular degree of the Kauffman bracket}~\label{KBSection} 

In this section we show that for a subset of the families of links we have considered previously that not only is the annular Khovanov homology non-zero in the top annular grading but the Kauffman bracket is as well. The main technique we use to show that the Kauffman bracket is non-zero for this subset is to work diagramatically and demonstrate that there is a quantum grading with exactly one generator in the annular Khovanov chain complex and that this generator is in the maximal annular grading. The generator in question will always be the generator obtained by taking a 1-resolution at each crossing and labeling every circle with a ``+".

This generator is in the maximal quantum grading for the chain complex and it has been observed that any other generator in the maximal quantum grading must come from a resolution that be connected to the all 1's resolution by a path of changes of resolutions where any time a 0-resolution changes to a 1-resolution two distinct circles merge together. Alternatively working backwards from the all 1's resolution changing a 1-resolution back to a 0-resolution must result in the splitting of a circle~\cite[Proposition~1]{gonzalez-meneses_geometric_2018}.

\begin{thm}\label{KBthm}
Let $L^n$ be the annular link built by vertically stacking $n$ copies of tangle $J$ from Figure~\ref{tangle} and then taking the annular closure. Also, let $L^{n}_{m,\beta_1, \ldots, \beta_n}$ be the link built by replacing the $i$-th component of $L^n$ with the closure of the $m$-braid $\beta_i$. Then the wrapping conjecture holds for the cable $L^{n}_{m,\beta_1, \ldots, \beta_n}$.
\end{thm}

The proof of Theorem~\ref{KBthm} is broken up into two parts. In the first two part of the proof we show the wrapping conjecture holds for the blackboard framed $m$-cable of $L^n$ by identifying a specific tangle and computing its all 1's resolution. Then to finish the proof we apply a result of Hoste-Przytycki.

\begin{proof}

\begin{figure}

  \centering
    \includegraphics[width=0.3\textwidth]{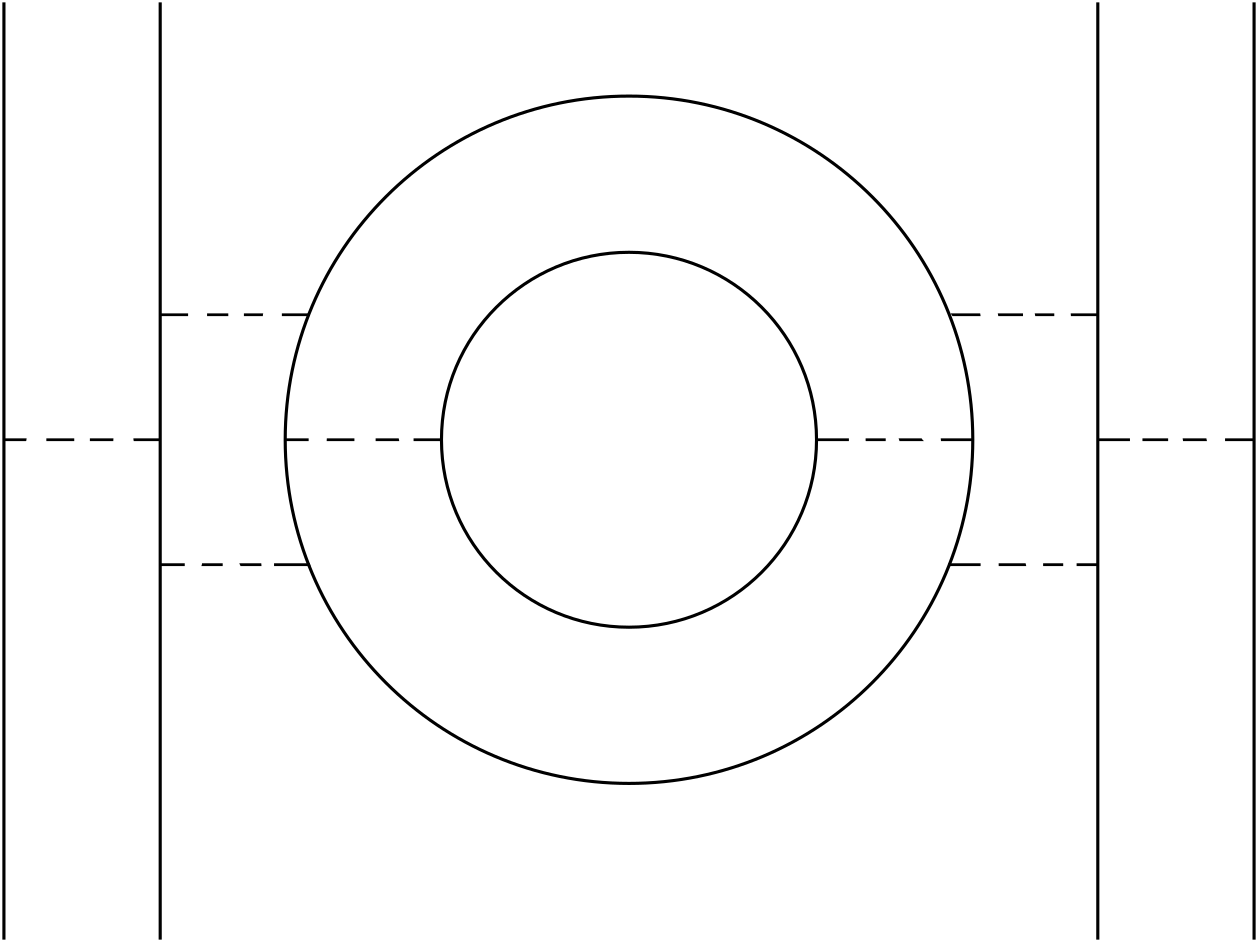}
      \caption{The resolution of the balanced tangle obtained from the 2-cable $L_2$. The dotted lines indicate where to attach bands to change back to a 0-resolution.}\label{2cableRes}
\end{figure}

\emph{Step~1:} We first compute the all 1's resolution for the sutured Khovanov homology of the balanced tangle whose annular closure is the annular link $L^{1}_{m, 1, \ldots, 1}$ where all the braids are the identity braid. For the 2-cable, the all 1's resolution is shown in Figure~\ref{2cableRes}. The dashed lines represent the bands that would be added to change a single 1-resolution back to a 0-resolution.

We will verify that for every $m$, the resolution in question is a tangle with $m$ strands running from top to bottom on the left followed by $m$ concentric circles followed by $m$ strands runnings from top to bottom on the right. We will also show that all the bands that should be attached to change a 1-resolution back to a 0-resolution are either between adjacent strands, between adjacent circles, or between the outermost circle and one of the strands adjacent to it.

Notice that the diagram for the sutured Khovanov homology of the $m$-cable can be viewed as a combination of the crossings involving either of the outer most strands and then the crossings that only involve the inner $m-1$ strands. The crossings involving only the inner $m-1$ cable correspond exactly to a diagram for the sutured Khovanov homology of the $m-1$-cable. We will consider resolving each of these halves individually and then place them together.

Considering the outer half and resolving all the crossings with the 1-resolution gives a tangle that is pictured for the $m$-cable in Figure~\ref{MidComputation}. The dashed lines indicate where to attach a band to go back to a $0$-resolution and the bold square represents where to place the inner half of the tangle. Notice that the bands all connect two different components of the tangle.

\begin{figure}

  \centering
    \includegraphics[width=0.55\textwidth]{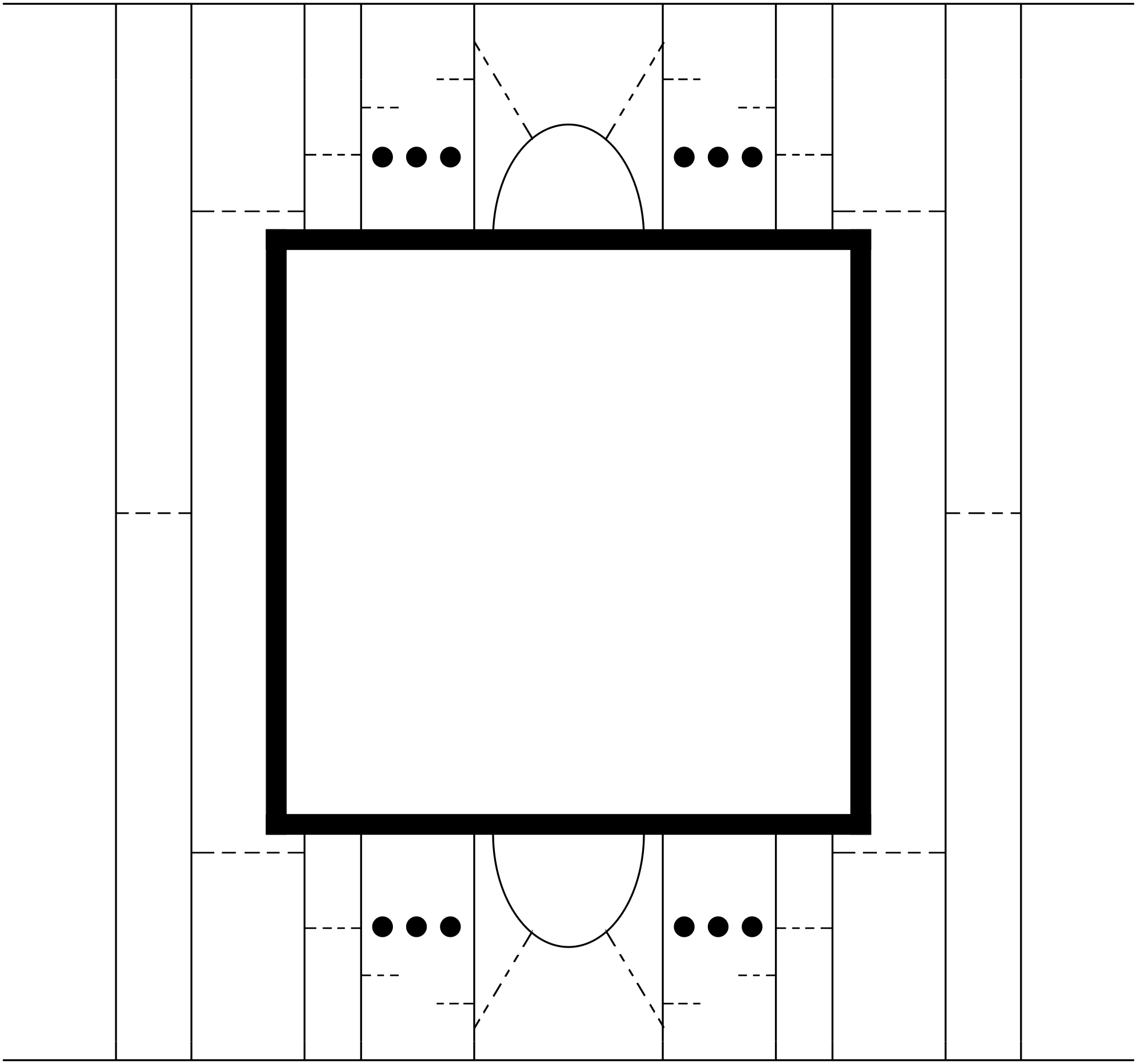}
      \caption{The all 1's resolution of the outer half of the tangle. The dashed lines indicate where to attach a band to go back to a $0$-resolution and the bold square represents where to place the inner half of the tangle.}\label{MidComputation}
\end{figure}

We can obtain the resolution of the $m$-cable by taking the resolved tangle from the previous paragraph and filling the middle hole with the resolution of the $m-1$-cable. Doing this gives a tangle with $m$ strands running from top to bottom on the left followed by $m$ concentric circles followed by $m$ strands runnings from top to bottom on the right. Furthermore, all the bands that should be attached to change a 1-resolution back to a 0-resolution are either between adjacent strands, between adjacent circles, or between the outermost circle and one of the two strands adjacent to the circle.

Now the all 1's resolution of the link $L^n_{m, 1, \ldots, 1 }$ can be obtained by stacking $m$ copies of the resolution considered above and taking the annular closure. Doing this gives a resolution that consists of $2m$ circles running around the annular axis and $n$ groups of $m$ concentric homotopically trivial circles. Furthermore all of the bands we would attach to revert a $1$-resolution back to a $0$-resolution run between distinct circles. This guarantees that the generator where all of these circles are labeled with a ``+" is the only generator of the chain complex in its quantum grading. It also sits in the maximal annular grading of $2m$ showing that the Kauffman bracket is non-zero in this annular grading and that the wrapping conjecture holds for the link $L^n_{m, 1, \ldots, 1 }$. In the language of~\cite{hoste_2_1995} we have shown that these links have minus-adequately wrapped diagrams.

\emph{Step~2:} A diagram for the link $L^{n}_{m,\beta_1, \ldots, \beta_n}$ can be obtained from a diagram for the link $L^n_{m, 1, \ldots, 1 }$ by the addition of the braids $\beta_i$ in the appropriate places in the diagram. Hoste-Przytycki give an argument showing that starting with a minus-adequately wrapped diagram and adding in a braid $\beta$ produces a new link for which the wrapping conjecture also holds~\cite[Lemma~10]{hoste_2_1995}. Repeated applications of this argument show that the wrapping conjecture holds for the link $L^{n}_{m,\beta_1, \ldots, \beta_n}$.
\end{proof}

\begin{thm}\label{KBWhThm}

Let $K^n$ be the iterated Whitehead double of the annular link $L^1$ obtained as the annular closure of the tangle $J$ from Figure~\ref{tangle} and let $K^n_{m,\beta}$ be the annular knot obtained by replacing $K^n$ with the closure of the $m$-braid $\beta$. Then the wrapping conjecture holds for $K^n_{m,\beta}$.

\end{thm}
\begin{proof}
 Notice that a diagram for $K^n_{m,1}$ can be built by stacking blackboard framed cables of the tangle $J$ from Figure~\ref{tangle} along with trivial braids to the left or right of these tangles and then taking an annular closure. This observation and the arguments from the proof of Theorem~\ref{KBthm} show that the generator from all 1's resolution with every circle marked with a $+$ sign is the only generator in its quantum grading and it is in the appropriate $k$-degree so the wrapping conjecture holds for $K^n_{m,1}$. In other words, there is a minus-adequately wrapped diagram for $K^n_{m,1}$. Applying~\cite[Lemma~10]{hoste_2_1995} ensures that the wrapping conjecture also holds for $K^n_{m,\beta}$.

\end{proof}

\bibliography{AKhTop}

\end{document}